\documentclass[11pt]{amsart}

\usepackage{amscd,amssymb,amsopn,amsmath,amsthm,mathrsfs,graphics,amsfonts,enumerate,verbatim,calc
}

\usepackage[all]{xy}

\usepackage{color}

\usepackage[OT2,OT1]{fontenc}
\newcommand\cyr{%
\renewcommand\rmdefault{wncyr}%
\renewcommand\sfdefault{wncyss}%
\renewcommand\encodingdefault{OT2}%
\normalfont
\selectfont}
\DeclareTextFontCommand{\textcyr}{\cyr}

\usepackage{amssymb,amsmath}

\DeclareFontFamily{OT1}{rsfs}{}
\DeclareFontShape{OT1}{rsfs}{n}{it}{<-> rsfs10}{}
\DeclareMathAlphabet{\mathscr}{OT1}{rsfs}{n}{it}

\numberwithin{equation}{section}
\newtheorem{theorem}{Theorem}[section]
\newtheorem{lemma}[theorem]{Lemma}
\newtheorem{proposition}[theorem]{Proposition}
\newtheorem{corollary}[theorem]{Corollary}

\theoremstyle{definition}

\newtheorem{remark}[theorem]{Remark}
\theoremstyle{remark}

\newtheorem{acknowledgement}{Acknowledgement}

\begin{document}
\title[On the absolute integral closure in positive characteristic]{On the vanishing of local cohomology
of \\the absolute integral closure in positive characteristic}

\author[Pham Hung Quy]{Pham Hung Quy}
\address{Department of Mathematics, FPT University, Hoa Lac Hi-Tech Park, Ha Noi, Viet Nam}
\email{quyph@fpt.edu.vn}

\thanks{
This work is partially supported by a fund of Vietnam National Foundation for Science
and Technology Development (NAFOSTED) under grant number
101.04-2014.25.}



\maketitle

\begin{abstract}
The aim of this paper is to extend the main result of C. Huneke and G. Lyubeznik in [Adv. Math.  210 (2007), 498--504] to the class of rings that are images of Cohen-Macaulay local rings. Namely, let $R$ be a local Noetherian domain of positive characteristic that is an image of a Cohen-Macaulay local ring. We prove that all local cohomology of $R$ (below the dimension) maps to zero in a finite extension of the ring. As a direct consequence we obtain that the absolute integral closure of $R$ is a big Cohen-Macaulay algebra. Since every excellent local ring is an image of a Cohen-Macaulay local ring, this result is a generalization of the main result of M. Hochster and Huneke in [Ann. of Math. 135 (1992), 45--79] with a simpler proof.
\end{abstract}

\section{Introduction}
Let $(R,\frak m)$ be a commutative Noetherian local domain with fraction field $K$. The {\it absolute integral closure} of $R$, denoted $R^+$, is the integral closure of $R$ in a fixed algebraic closure $\overline{K}$ of $K$.\\

A famous result of M. Hochster and C. Huneke says that if $(R, \frak m)$ is an excellent local Noetherian domain of positive characteristic $p > 0$, then $R^+$ is a (balanced) big Cohen-Macaulay algebra, i.e. every system of parameters in $R$ becomes a regular sequence in $R^+$ (cf. \cite{HH92}). Furthermore, K.E. Smith in \cite{Sm94} proved that the tight closure of an ideal generated by parameters is the contraction of its extension in $R^+$: $I^* = IR^+ \cap R$. This property is not true for every ideal $I$ in an excellent Noetherian domain since tight closure does not commute with localization (cf. \cite{BM10}).\\

As mentioned above, $H^i_{\frak m}(R^+) = 0$ for all $i< \dim R$ provided $R$ is an excellent local Noetherian domain of positive characteristic. Hence, the natural homomorphism $H^i_{\frak m}(R) \to H^i_{\frak m}(R^+)$, induced from the inclusion $R \to R^+$, is the zero map for all $i< \dim R$. In the case $R$ is an image of a Gorenstein (not necessarily excellent) local ring, as the main result of \cite{HL07}, Huneke and G. Lyubeznik proved a stronger conclusion that one can find a finite extension ring $S$, $R \subseteq S \subseteq R^+$, such that the natural map $H^i_{\frak m}(R) \to H^i_{\frak m}(S)$ is zero for all $i < \dim R$. Therefore, they obtained a simpler proof for the result of Hochster and Huneke in the cases where the assumptions overlap, e.g., for complete Noetherian local domain. The techniques used in \cite{HL07} are the Frobenius action on the local cohomology, (modified) equation lemma (cf. \cite{HH92}, \cite{Sm94}, \cite{HL07}) and the local duality theorem (This is the reason of the assumption that $R$ is an image of a Gorenstein local ring). The motivation of the present paper is our belief: {\it If a result was shown by the local duality theorem, then it can be proven under the assumption that the ring is an image of a Cohen-Macaulay local ring} (for example, see \cite{NQ14}). The main result of this paper extends Huneke-Lyubeznik's result to the class of rings  that are images of Cohen-Macaulay local rings. Namely, we prove the following.

\begin{theorem}\label{T1.1} Let $(R, \frak m)$ be a commutative Noetherian local domain containing a filed of positive characteristic $p$. Let $K$ be the fraction field of $R$ and $\overline{K}$ an algebraic closure of $K$. Assume that $R$ is an image of a Cohen-Macaulay local ring. Let $R'$ be an $R$-subalgebra of $\overline{K}$ (i.e. $R \subseteq R' \subseteq \overline{K}$) that is a finite $R$-module. Then there is an $R'$-subalgebra $R''$ of $\overline{K}$ (i.e. $R' \subseteq R'' \subseteq \overline{K}$) that is finite as an $R$-module such that the natural map $H^i_{\frak m}(R') \to H^i_{\frak m}(R'')$ is the zero map for all $i < \dim R$.
\end{theorem}
As a direct application of Theorem \ref{T1.1} we obtain that the absolute integral closure $R^+$ is a big Cohen-Macaulay algebra (cf. Corollary \ref{C3.3}). It worth be noted that every excellent local ring is an image
of a Cohen-Macaulay excellent local ring by T. Kawasaki (cf. \cite[Corollary 1.2]{K01}). Therefore, our results also extend the original result of Hochster and Huneke with a simpler proof. The main results will be proven in the last section. In the next section we recall the theory of attached primes of Artinian (local cohomology) modules.

\section{Preliminaries}

Throughout this section $(R, \frak m)$ be a commutative Noetherian local ring. We recall the main result of \cite{NQ14} which is an illustration for our belief (mentioned in the introduction).\\

I.G. Macdonald, in \cite{Mac73}, introduced the theory of secondary representation for Artinian modules, which is in some sense dual to the theory of primary decomposition for Noetherian modules. Let $A\neq 0$ be an Artinian $R$-module. We say that $A$ is {\it secondary} if the multiplication by $x$ on $A$ is surjective or nilpotent for every $x\in R$. In this case, the set $\frak p := \sqrt{(\mathrm{Ann}_RA)}$ is a prime ideal of $R$ and we say that $A$ is {\it $\frak p$-secondary}. Note that every Artinian $R$-module $A$ has a minimal secondary representation $A=A_1+\ldots+A_n,$ where $A_i$ is $\frak p_i$-secondary, each $A_i$ is not redundant and $\frak p_i\neq \frak p_j$ for all $i\neq j.$ The set $\{\frak p_1,\ldots,\frak p_n\}$ is independent of the choice of the minimal secondary representation of $A$. This set is called the set of {\it attached primes} of $A$ and denoted by $\mathrm{Att}_RA$. Notice that if $R$ is complete we have the Matlis dual $D(A)$ of $A$ is Noetherian and $\mathrm{Att}_RA = \mathrm{Ass}_RD(A)$.

For each ideal $I$ of $R$, we denote by $Var(I)$  the set of all prime ideals of $R$ containing $I$. The following is easy to understand from the theory of associated primes.

\begin{lemma}[\cite{Mac73}] \label{L2.1} Let $A$ be an Artinian $R$-module. The following statements are true.
\begin{enumerate}[{(i)}]\rm
\item {\it $A\neq 0$ if and only if $\mathrm{Att}_RA\neq \emptyset.$}
\item {\it $A \neq 0$ has finite length if and only if $\mathrm{Att}_RA\neq \{\frak m\}.$}
\item {\it $\min \mathrm{Att}_RA=\min Var(\mathrm{Ann}_RA).$ In particular,  $$\dim (R/\mathrm{Ann}_RA)=\max \{\dim(R/\frak p) \mid \frak p\in \mathrm{Att}_RA\}.$$}
\item {\it If $0\rightarrow A'\rightarrow A\rightarrow A''\rightarrow 0$ is an exact sequence of Artinian $R$-modules then
$$\mathrm{Att}_RA''\subseteq \mathrm{Att}_RA\subseteq \mathrm{Att}_RA'\cup\mathrm{Att}_RA''.$$}
 \end{enumerate}
\end{lemma}

Let $\widehat{R}$ be the $\frak m$-adic complete of $R$. Note that every Artinan $R$-module $A$ has a natural structure as an $\widehat{R}$-module and with this structure, each subset of $A$ is an $R$-submodule if and only if it is an $\widehat{R}$-submodule. Therefore $A$ is an Artinian $\widehat{R}$-module. So, the set of attached primes $\mathrm{Att}_{\widehat{R}}A$ of $A$ over $\widehat{R}$ is well defined.

\begin{lemma}\label{L2.2} (\cite[8.2.4, 8.2.5]{BS98}).  $\mathrm{Att}_RA=\big\{P\cap R \mid P\in\mathrm{Att}_{\widehat{R}}A\big\}.$
\end{lemma}


Let $M$ be a finitely generated $R$-module. It is well known that the local cohomology module $H^i_{\frak m}(M)$ is Artinian for all $i\geq 0$ (cf. \cite[Theorem 7.1.3]{BS98}). Suppose that $R$ is an image of a Gorenstein local ring. R.Y. Sharp, in \cite{Sh75}, used the local duality theorem to prove the following relation
$$ \mathrm{Att}_{R_{\frak p}}\big(H^{i-\dim (R/\frak p)}_{\frak p R_{\frak p}}(M_{\frak p})\big)=\big\{\frak q R_{\frak p}\mid \frak q\in\mathrm{Att}_R(H^i_{\frak m}(M)), \frak q \subseteq \frak p\big\}$$
for all $\frak p \in \mathrm{Supp}(M)$ and all $i \ge 0$. Based on the study of splitting of local cohomology (cf. \cite{CQ11}, \cite{CQ13}), L.T. Nhan and the author showed that the above relation holds true on the category of finitely generated $R$-modules if and only if $R$ is an image of a Cohen-Macaulay local ring (cf. \cite{NQ14}). It worth be noted that $R$ is an image of a Cohen-Macaulay local ring if and only if $R$ is universally catenary and all its formal fibers are Cohen-Macaulay by Kawasaki (cf. \cite[Corollary 1.2]{K01}). More precisely, we proved the following.

\begin{theorem} \label{T2.4} The following statements are equivalent:
\begin{enumerate}[{(i)}]\rm
\item {\it $R$ is an image of a Cohen-Macaulay local ring;}
\item {\it $ \mathrm{Att}_{R_{\frak p}}\big(H^{i-\dim (R/\frak p)}_{\frak p R_{\frak p}}(M_{\frak p})\big)=\big\{\frak q R_{\frak p}\mid \frak q\in\mathrm{Att}_R(H^i_{\frak m}(M)), \frak q \subseteq \frak p\big\}$
 for every finitely generated $R$-module $M$, integer $i\ge 0$ and prime ideal $\frak p$ of $R$;}
\item {\it $\displaystyle\mathrm{Att}_{\widehat{R}}(H^i_{\frak m}(M))=\bigcup_{\frak p\in\mathrm{Att}_R(H^i_{\frak m}(M))}\mathrm{Ass}_{\widehat{R}}(\widehat{R}/\frak p\widehat{R})$ for every finitely generated $R$-module $M$ and  integer $i\geq 0$.}
  \end{enumerate}
\end{theorem}
The above Theorem says that the attached primes of local cohomology modules have good behaviors with completion and localization when $R$ is an image of a Cohen-Macaulay local ring. This will be very useful in the next section.

\section{Proof the main result}
Throughout this section, let $(R, \frak m, k)$ be a commutative Noetherian local ring that is an image of a Cohen-Macaulay local ring. The following plays the key role in our proof of the main result.

\begin{proposition}\label{P3.1} Let $M$ and $N$ be finitely generated $R$-modules and $\varphi: M \to N$ a homomorphism. For each $i \ge 0$, $\varphi$ induces the homomorphism $\varphi^i : H^i_{\frak m}(M) \to H^i_{\frak m}(N)$. Suppose for all $\frak p \in \mathrm{Att}_R(H^i_{\frak m}(M))$ and $\frak p \neq \frak m$, the map $\varphi_{\frak p}: M_{\frak p} \to N_{\frak p}$ induces the zero map
$$\varphi^{i -t_{\frak p}}_{\frak p}: H^{i-t_{\frak p}}_{\frak p R_{\frak p}}(M_{\frak p}) \to H^{i-t_{\frak p}}_{\frak p R_{\frak p}}(N_{\frak p}),$$
 where $t_{\frak p} =\dim R/\frak p$. Then $\mathrm{Im}(\varphi^i)$ has finite length.
\end{proposition}
\begin{proof}
Suppose $\mathrm{Im}(\varphi^i)$ has not finite length. By Lemma \ref{L2.1} there exists $\frak m \neq \frak p \in \mathrm{Att}_R(\mathrm{Im}(\varphi^i))$. So $\frak p \in \mathrm{Att}_R (H^i_{\frak m}(M))$ by Lemma \ref{L2.1} (iv). Consider $\mathrm{Im}(\varphi^i)$ as an Artinian $\widehat{R}$-module. By Lemma \ref{L2.2}, there exists $P \in \mathrm{Att}_{\widehat{R}}(\mathrm{Im}(\varphi^i))$ such that $P \cap R = \frak p$. Hence we have $P \in \mathrm{Att}_{\widehat{R}}(H^i_{\frak m}(M))$ by Lemma \ref{L2.1} (iv) again. Since $R$ is an image of a Cohen-Macaulay local ring, Theorem \ref{T2.4} (iii) implies that $P \in \mathrm{Ass}_{\widehat{R}} (\widehat{R}/\frak p \widehat{R})$. Therefore $\dim \widehat{R}/P = \dim R/\frak p$ by \cite[Theorem 2.1.15]{BH98}. We have $\widehat{R}$ is complete, so it is an image of a Gorenstein local ring $S$ (of dimension $n$). By local duality we have
$$D(\mathrm{Ext}^{n-i}_S(\widehat{M}, S)) \cong H^i_{\widehat{\frak m}} (\widehat{M}) \quad (\cong H^i_{\frak m}(M) \otimes_R \widehat{R} \cong H^i_{\frak m}(M)) ,$$
where $D = \mathrm{Hom}_{\widehat{R}} (-, E_{\widehat{R}}(k))$ is the Matlis duality functor (cf. \cite[Theorem 11.2.6]{BS98}). Since $\widehat{R}$ is complete we have $\mathrm{Ext}^{n-i}_S(\widehat{M}, S) \cong D(H^i_{\frak m}(M))$.\\
We write the map $\varphi^i : H^i_{\frak m}(M) \to H^i_{\frak m}(N)$ as the composition of two maps
$$H^i_{\frak m}(M) \to \mathcal{I} = \mathrm{Im}(\varphi^i) \to H^i_{\frak m}(N),$$
where the first of which is surjective and the second injective. Applying the Matlis duality functor $D$ we get the map $D(\varphi^i): \mathrm{Ext}^{n-i}_S(\widehat{N}, S) \to \mathrm{Ext}^{n-i}_S(\widehat{M}, S)$ is the composition of two maps
$$\mathrm{Ext}^{n-i}_S(\widehat{N}, S) \to D(\mathcal{I}) \to \mathrm{Ext}^{n-i}_S(\widehat{M}, S)$$
with the first of which is surjective and the second injective. We have $D(\mathcal{I})$ is a finitely generated $\widehat{R}$-module and $P \in
\mathrm{Ass}_{\widehat{R}}D(\mathcal{I}) \,\, (= \mathrm{Att}_{\widehat{R}}\mathcal{I})$. Let $P'$ be the pre-image of $P$ in $S$. Localization at $P'$ the above composition we get the composition
$$\mathrm{Ext}^{n-i}_{S_{P'}}(\widehat{N}_{P}, S_{P'}) \to (D(\mathcal{I}))_P \to \mathrm{Ext}^{n-i}_{S_{P'}}(\widehat{M}_P, S_{P'})$$
with the first of which is surjective and the second injective. Since $(D(\mathcal{I}))_P \neq 0$, we have the map
$$D(\varphi^i)_{P}: \mathrm{Ext}^{n-i}_{S_{P'}}(\widehat{N}_{P}, S_{P'}) \to \mathrm{Ext}^{n-i}_{S_{P'}}(\widehat{M}_P, S_{P'})$$
is a non-zero map. Notice that $\dim S_{P'} = n - t_{\frak p}$. Applying local duality (for $S_{P'}$) we have the map
$$\widehat{\varphi}^{i-t_{\frak p}}_P : H^{i-t_{\frak p}}_{P\widehat{R}_P}(\widehat{M}_P) \to H^{i-t_{\frak p}}_{P\widehat{R}_P}(\widehat{N}_P),$$
induced from the map $\widehat{\varphi}: \widehat{M} \to \widehat{N}$, is a non-zero map. Recalling our assumption that the map
$$\varphi^{i -t_{\frak p}}_{\frak p}: H^{i-t_{\frak p}}_{\frak p R_{\frak p}}(M_{\frak p}) \to H^{i-t_{\frak p}}_{\frak p R_{\frak p}}(N_{\frak p}),$$
induced from $\varphi : M \to N$, is zero.\\

On the other hand, the faithfully flat homomorphism of local rings $(R, \frak m) \to (\widehat{R}, \widehat{\frak m})$ induces the flat homomorphism $(R_{\frak p}, \frak p R_{\frak p}) \to (\widehat{R}_P, P\widehat{R}_P)$ by \cite[Theorem 7.1]{Mat86}. It is a local homomorphism so we have a faithfully flat homomorphism. It should be noted that $\sqrt{\frak p\widehat{R}_P} = P\widehat{R}_P$. Using the following commutative diagram of flat homomorphisms
\[
\xymatrix{
R \ar[d] \ar[r]      & R_{\frak p} \ar[d]\\
\widehat{R} \ar[r]   & \widehat{R}_P
          }
\]
one can check that
$$H^{i-t_{\frak p}}_{P\widehat{R}_P}(\widehat{M}_P) \cong H^{i-t_{\frak p}}_{\frak p R_{\frak p}}(M_{\frak p}) \otimes_{R_{\frak p}} \widehat{R}_P$$
and $\widehat{\varphi}^{i-t_{\frak p}}_P$ is just $\varphi^{i -t_{\frak p}}_{\frak p} \otimes_{R_{\frak p}} \widehat{R}_P$.
Therefore the maps $\varphi^{i -t_{\frak p}}_{\frak p}$ and $\widehat{\varphi}^{i-t_{\frak p}}_P$ are either zero or non-zero, simultaneously. This is a contradiction. The proof is complete.
\end{proof}
We are ready to prove the main result of this paper. In the rest of this section, $(R, \frak m)$ is a local domain of positive characteristic $p$ that is an image of a Cohen-Macaulay local ring. Let $I$ be an ideal of $R$ and $R'$ an $R$-algebra. Notice that the local cohomology, $H^i_I(-)$, can be computed via the \v{C}ech co-complex of the generators of $I$. The Frobenius ring homomorphism
$$f: R' {\longrightarrow} R'; {r \mapsto r^p }$$
induces a natural map $f_* : H^i_I(R') \to H^i_I(R')$ on all $i \ge 0$. It is called the (natural) action of Frobenius on $H^i_I(R')$.

\begin{proof}[Proof of Theorem \ref{T1.1}] We proceed by induction on $d = \dim R$. There is nothing to prove when $d=0$. Assume that $d>0$ and the theorem is proven for all smaller dimension. For each $i < d$ and $\frak p \in \mathrm{Att}_RH^i_{\frak m}(R)$, $\frak p \neq \frak m$, by the inductive hypothesis there is an $R'_{\frak p}$-subalgebra $\widetilde{R}^{i,\frak p}$ that is finite as $R_{\frak p}$-module such that the natural map
$$H^{i-t_{\frak p}}_{\frak pR_{\frak p}}(R'_{\frak p}) \to H^{i-t_{\frak p}}_{\frak pR_{\frak p}}(\widetilde{R}^{i,\frak p})$$
is the zero map, where $t_{\frak p} = \dim R/\frak p$. Let $\widetilde{R}^{i,\frak p} = R'_{\frak p}[z_1, ..., z_k]$, where $z_1, \ldots, z_k \in \overline{K}$ are integral over $R_{\frak p}$. Multiplying, if necessary, some suitable element of $R \setminus \frak p$, we can assume that each $z_j$ is integral over $R$. Set $R^{i, \frak p} = R'[z_1, \ldots, z_k]$. Clearly, $R^{i, \frak p}$ is an $R'$-subalgebra of $\overline{K}$ that is finite as $R$-module.\\
Since the sets $\{i \mid 0 \le i <d\}$ and $\mathrm{Att}_R(H^i_{\frak m}(R))$ are finite, the following is a finite extension of $R$
$$R^* = R[R^{i,\frak p} \mid i<d, \frak p \in \mathrm{Att}_R(H^i_{\frak m}(R)) \setminus \{\frak m\}].$$
  We have $R^*$ is an $R^{i, \frak p}$-subalgebra of $\overline{K}$ for all $i<d$ and all $\frak p \in \mathrm{Att}_R(H^i_{\frak m}(R)) \setminus \{\frak m\}$. The inclusions $R' \to R^{i,\frak p} \to R^*$ induce the natural maps
$$H^{i-t_{\frak p}}_{\frak pR_{\frak p}}(R'_{\frak p}) \to H^{i-t_{\frak p}}_{\frak pR_{\frak p}}(\widetilde{R}^{i,\frak p}) \to H^{i-t_{\frak p}}_{\frak pR_{\frak p}}(R^*_{\frak p}).$$
By the construction of $\widetilde{R}^{i,\frak p}$ we have the natural map
$$H^{i-t_{\frak p}}_{\frak pR_{\frak p}}(R'_{\frak p})  \to H^{i-t_{\frak p}}_{\frak pR_{\frak p}}(R^*_{\frak p})$$
is the zero map for all $i<d$ and all $\frak p \in \mathrm{Att}_RH^i_{\frak m}(R) \setminus \{\frak m\}$. By Proposition \ref{P3.1} we have the natural map
$$\varphi^i : H^i_{\frak m}(R') \to H^i_{\frak m}(R^*),$$
induced from the inclusion $\varphi : R' \to R^*$, has $\ell( \mathrm{Im}(\varphi^i))< \infty$ for all $i<d$.\\

Since the natural inclusion $\varphi : R' \to R^*$ is compatible with the Frobenius homomorphism on $R'$ and $R^*$, we have $\varphi^i$ is compatible with the Frobenius action $f_*$ on $H^i_{\frak m}(R')$ and $H^i_{\frak m}(R^*)$. Therefore $\mathrm{Im} \varphi^i$ is an $f_*$-stable $R$-submodule of $H^i_{\frak m}(R^*)$, i.e. $f_*(\alpha) \in \mathrm{Im} \varphi^i$ for every $\alpha \in \mathrm{Im} \varphi^i$. Since $\mathrm{Im} \varphi^i$ has finite length, every $\alpha \in \mathrm{Im} \varphi^i$ satisfies the "equation lemma" of Huneke-Lyubeznik (cf. \cite[Lemma 2.2]{HL07}). Hence each element of $\mathrm{Im} \varphi^i$ will map to the zero in local cohomology of a certainly finite extension of $R^*$. Since $\mathrm{Im} \varphi^i$ is a finitely generated $R$-module for all $i<d$, there is an $R^*$-subalgebra $R''$ of $\overline{K}$ that is finite as $R$-module such that the composition of the natural maps $H^i_{\frak m}(R') \to H^i_{\frak m}(R^*)  \to H^i_{\frak m}(R'')$ is zero for all $i<d$. The proof is complete.
\end{proof}

Similar to \cite[Corollary 2.3]{HL07} we have the following.
\begin{corollary}\label{C3.3}
 Let $(R, \frak m)$ be a commutative Noetherian local domain containing a field of positive characteristic $p$ and $R^+$ the absolute integral closure of $R$ in $\overline{K}$. Assume that $R$ is an image of a Cohen-Macaulay local ring. Then the following hold:
 \begin{enumerate}[{(i)}]\rm
\item {\it $H^i_{\frak m}(R^+) = 0$ for all $i < \dim R$}.
\item {\it Every system of parameters of $R$ is a regular sequence on $R^+$, i.e. $R^+$ is a big Cohen-Macaulay algebra.}
\end{enumerate}
\end{corollary}
We close this paper with the following.
\begin{remark}\rm
 \begin{enumerate}[{(i)}]
\item In \cite{SS12}, A. Sannai and A.K. Singh showed that the finite extension in Huneke-Lyubeznik's result can be chosen as a generically Galois extension. It is not difficult to see that our method also works for Sannai-Singh's paper. Hence Theorem 1.3 (2) and Corollary 3.3 of \cite{SS12} hold true when the ring is an image of a Cohen-Macaulay local ring.
\item Since an excellent local ring is an image of a Cohen-Macaulay excellent local ring (\cite[Corollary 1.2]{K01}), Corollary 3.3 is a generalization of the original result of Hochster and Huneke in \cite{HH92} with a simpler proof. Thus our results cover all previous results. On the other hand, R.C. Heitmann constructed examples of universally catenarian local domains with the absolute integral closures are not Cohen-Macaulay (cf. \cite[Corollary 1.8]{He10}). Comparing with the result of Kawasaki (\cite[Corollary 1.2]{K01}), the condition that $R$ is an image of a Cohen-Macaulay local ring is seem to be the most general case for $R^+$ is big Cohen-Macaulay.
\end{enumerate}
\end{remark}

\begin{acknowledgement}
The author is grateful to Professor Kei-ichi Watanabe for his helpful recommendations. He is also grateful to Kazuma Shimomoto for his interest and many discussions.
\end{acknowledgement}

\end{document}